\documentclass{birkjour}

\def\Q{\mathbb Q}
 
\newtheorem{thm}{Theorem}[section]
\newtheorem{cor}[thm]{Corollary}
\newtheorem{lem}[thm]{Lemma}

\theoremstyle{definition}

\theoremstyle{remark}
\newtheorem{rem}[thm]{Remark}

\numberwithin{equation}{section}

\begin{document}
\title[The first factor of the class number of the $p$-th cyclotomic field]
{The first factor of the class number of the $p$-th cyclotomic field}
\author[Korneel Debaene]{Korneel Debaene}
\address{
Krijgslaan 281, S22\\
9000 Ghent\\
Belgium}
\email{kdbaene@cage.ugent.be}

\begin{abstract}
Kummer's conjecture states that the relative class number of the $p$-th cyclotomic field follows a strict asymptotic law. Granville has shown it unlikely to be true -- it cannot be true if we assume the truth of two other widely believed conjectures. We establish a new bound for the error term in Kummer's conjecture, and more precisely we prove that $\log(h_p^-)=\frac{p+3}{4}\log p + \frac{p}{2}\log(2\pi)+\log(1-\beta)+O(\log_2 p)$, where $\beta$ is a possible Siegel zero of an $L(s,\chi)$, $\chi$ odd. \end{abstract}


\maketitle

\section{Introduction}
Let $h_p$ denote the class number of $\Q(\zeta_p)$, where $p$ is an odd prime. Let $h_p^+$ denote the class number of the totally real field $\Q(\zeta_p+\zeta_p^{-1})$. It is well known that $h_p^+$ divides $h_p$. We denote the quotient --- or so called {\em first factor} of $h_p$ --- by $h_p^-$.
The following formula is an application of the class number formula(see e.g. \cite{Washington})
\begin{align}\label{eq:classnumberformula}h_p^- = G(p)\prod_{\textup{$\chi$ mod $p$,odd}}L(1,\chi),\end{align}
where $G(p) = 2p\left(\frac{p}{4\pi^2}\right)^{\frac{p-1}{4}}$. Since the magnitude of the $L(1,\chi)$ is not evident, it is hoped they are insignificant. The guess that $h_p^-$ is asymptotically equivalent to $G(p)$ is known as Kummer's Conjecture. It is opportune to study the logarithm of this equation because the orthogonality property of characters gives us \begin{align}\label{eq:orthog}\sum_{\textup{$\chi$ mod $p$, odd}}\log(L(s,\chi))=\frac{p-1}{2}\left(\sum_{q^m\equiv 1(p)}\frac{1}{mq^{ms}}-\sum_{q^m\equiv -1(p)}\frac{1}{mq^{ms}}\right).\end{align}
One can estimate this sum to the right of $s=1$, where good estimates are available, and using a zero-free region of the $L$-functions, one can bound the derivative in a neighbourhood of $s=1$. Masley and Montgomery \cite{MM} obtained with these key ingredients that $|\log(h_p^-/G(p))|<7\log p$ for $p>200$, which is strong enough to solve the class number one problem for cyclotomic fields. 

Puchta \cite{Puchta} improved this approach by using analogous bounds on higher derivatives, and using a {\em near} zero-free region, namely the open ball $B(1, \frac{1}{c\log p})$ with center 1 and radius $\frac{1}{c\log p}$, where $c$ is some big enough constant. This is a zerofree region for all but possibly one $L$-function mod $p$, which then is necessarily quadratic and has one zero $\beta$ in this region, which is necessarily real and simple and goes by the name of a Siegel zero. It is worth mentioning that if $p=1$ mod $4$, the odd characters are not quadratic, hence have no Siegel zero.
Puchta obtained $\log(h_p^-/G(p))=\log(1-\beta)+O((\log_2 p)^2)$. 

Our proof will follow the main ideas from \cite{Puchta}, but our practical implementation in section \ref{WorstCase} is of a different nature, and yields \begin{thm}\label{Theorem1}If no Siegel zero is present among the odd Dirichlet $L$-functions of conductor $p$, then the relative class number of $\Q(\zeta_p)$ satisfies \[|\log(h_p^-/G(p))|\leq 2\log_2(p) + O(\log_3(p))\]
If there is a Siegel zero $\beta$ present among the odd Dirichlet $L$-functions of conductor $p$, then the relative class number of $\Q(\zeta_p)$ satisfies
\[|\log(h_p^-/G(p))-\log(1-\beta)|\leq 4\log_2(p) + O(\log_3(p))\]\end{thm} 
\noindent Since $\log(1-\beta)$ is negative, an upper bound without this term may be deduced. Finally, we note that this result sharpens the best known estimate, by Lepist\"o \cite{Lepisto}. Indeed, he proves an upper bound for $\log(h_p^-/G(p))$ with main term $5\log_2(p)$.

\section{Bounds around $s=1$}
In this section we exploit formula $\eqref{eq:orthog}$, which gives a representation in terms of splitting behaviour in $\Q(\zeta_p) / \Q$. We define \[ \Pi(x,p,a) = \sum_{q^m \leq x, q^m\equiv a (p)}\frac{1}{mq^m},\] where $q^m$ ranges over the primepowers. A Brun-Titschmarsh style bound is given by the following lemma.

\begin{lem} For $x>p$, and $p>500$ we have that \[\Pi(x,p,\pm 1) \leq \frac{2x}{(p-1)\log(x/p)}\]
\end{lem} 
\begin{proof}
When $x\geq p^2$, we start from the following inequality (see \cite{MM}, Lemma 1) \[\Pi(x,p,\pm 1)\leq \pi(x,p,\pm 1) +\frac{4\sqrt{x}}{p}+\log x.\]
In \cite{montgompaper}, the following strong version of the Brun-Titchmarsh inequality is proven. \[\pi(x,p,\pm 1) \leq \frac{2x}{(p-1)(\log(x/p)+5/6)}\] Thus we only need to prove that \[\frac{4\sqrt{x}}{p}+\log x < \frac{2x}{(p-1)}\left(\frac{1}{\log(x/p)} - \frac{1}{\log(x/p)+5/6} \right).\]
By setting $x=pX$, $X\geq p$, it suffices to prove that \[g(X):=\frac{4}{\sqrt{p}} + \frac{\log(pX)}{\sqrt{X}} < h(X):=\frac{5\sqrt{X}}{3(\log X + 5/6)^2}.\]
Now, $g(X)$ decreases for $X-e^2$ and $h(X)$ increases for $X\geq e^{19/6}$, hence it suffices to check that \[g(p)=\frac{4}{\sqrt{p}} + \frac{2\log(p)}{\sqrt{p}} < h(p)=\frac{5\sqrt{p}}{3(\log p + 5/6)^2}\]
for $p\geq 500$. Now, $g(p)$ decreases for $p\geq 2$ and $h(p)$ increases for $p\geq e^{19/6}$, hence it suffices to check that $g(500)<h(500)$, which is clear.

When $p<x<p^2$, any two primepowers in the sum $\Pi(x,p,\pm 1)$ are necessarily coprime. Indeed, their quotient would be $1\mod p$, so at least $p+1$, implying that the smallest one should be less than $\frac{p^2}{p+1}$. The only option then is that $p-1=2^m$ and $p^2-1=2^k$, but except for $p=3$ this is impossible. Thus, $\Pi(x,p,\pm 1)\leq N(x,Q,p,\pm 1) + \pi(Q)$, where $N(x,Q,p,a)$ is the number of integers $n\equiv a$ (mod $p$), $n\leq x$ such that $n$ is not divisible by any prime number less then $Q$. We may bound $\pi(Q)$ trivially by $Q$, so that the quantity to be bounded is $N(x,Q,p,\pm 1) + Q$. 

In the proof of the Brun-Titchmarsh inequality \[\pi(x,q,\pm 1)\leq \frac{2x}{(p-1)\log(x/p)}\] using the large sieve, as in \cite[p.42-44]{montgom}, the first step is to bound $\pi(x,q,\pm1)$ by exactly the quantity $N(x,Q,p,\pm 1) + Q$. This shows that in this range of $x$, the large sieve method for the Brun-Titchmarsh inequality can be applied with the same success for primepowers as for primes.
\end{proof}

\noindent Let us define $f(s)$ by \[f(s) = \Big(\sum_{\chi(-1)=-1}\log L(s,\chi)\Big) - \log(s-\beta),\] in case that any of the $L$-functions with $\chi$ odd has a so-called Siegel zero $\beta$ in $]1-\frac{1}{c\log p},1]$, where $c$ is some big enough constant. Otherwise, we leave out the term with the Siegel zero. In any case $f$ is holomorphic in $B(1,\frac{1}{c\log p})$.

\begin{lem} For any $c$, $p\geq 500$, and $\sigma \in ]1,1+\frac{1}{c\log p}]$, we have the following estimates.
\begin{align}&\label{eq:divergingbasic}|f(\sigma)|\leq (1+1_\beta)\log\big(\frac{1}{\sigma-1}\big) + \frac{3}{2}\\
\label{eq:diverging}&|f^{(\nu)}(\sigma)| \leq \big(1+1_\beta+c_{p,\nu}\big)\frac{(\nu-1)!}{(\sigma-1)^\nu}
\end{align}
Where the notation $1_\beta$ stands for $1$ if a Siegel zero is present and $0$ otherwise, and we may choose the $c_{p, \nu}$ to be equal to $\frac{\log(2)}{2c^\nu (\nu-1)!\log p}
+\frac{\log_2(p)+\log(c)-\log_2(2) + e^{-1}}{c^\nu (\nu-1)!}
+\frac{1}{c\log p} + \frac{\sigma \lfloor \log \nu \rfloor}{\nu - \lfloor \log \nu \rfloor} + \frac{\sigma\nu}{c^{\lfloor \log \nu \rfloor} \lfloor \log \nu \rfloor !}$.
\end{lem}
\begin{proof}
The case $\nu=0$ can be proven as in \cite{MM}. The estimates for the derivatives are stated in \cite{Puchta}, but the statement is slightly incorrect and the proof omitted, so we will prove them here in full. We bound the sums occurring in the $\nu$-th derivative of \eqref{eq:orthog} using Lemma 2.1 and partial summation. 
\begin{align*}\frac{p-1}{2}&\sum_{q^m \equiv 1 (p)}  \frac{(m\log q)^\nu}{mq^{m\sigma}} = \frac{p-1}{2}\int_{2p}^\infty \frac{(\log x)^\nu d(\Pi(x,p,1))}{x^\sigma}\\
&=\frac{p-1}{2}\int_{2p}^\infty \frac{\sigma x^{\sigma-1}(\log x)^{\nu} -\nu x^{\sigma-1}(\log x)^{\nu-1}}{x^{2\sigma}}\Pi(x,p,1) dx\\
&\leq\int_{2p}^\infty \frac{\sigma(\log x)^\nu}{x^{\sigma}\log(x/p)} dx\\
&=\frac{p\sigma}{p^\sigma }\int_{2}^\infty \frac{(\log x+\log p)^{\nu}}{x^{\sigma}\log x} dx=:I,\\
\intertext{where we possibly omitted the first term $\frac{(p-1)\log(p+1)^\nu}{2m(p+1)^\sigma}$ if $p+1$ is a primepower $q^m$. If this is the case, then $q=2$ and $m=\log(p+1)/\log(2)$. This term is smaller than $\varepsilon_1\frac{(\nu-1)!}{(\sigma-1)^\nu}$ for all $\sigma$ in the desired range for $\varepsilon_1 = \frac{\log(2)}{2c^\nu (\nu-1)!\log p}$. We expand the integrand with the binomial theorem, and get}
&I =  \frac{p\sigma}{p^\sigma}(\log p)^\nu\int_{2}^\infty \frac{1}{x^{\sigma}\log x} dx + \frac{p\sigma}{p^\sigma}\sum_{i=0}^{\nu-1}\frac{\nu!(\log p)^i}{(\nu-i)! i!}  \int_{1}^{\infty}\frac{(\log x)^{\nu-i-1}}{x^\sigma} dx\\
& = \frac{p\sigma}{p^\sigma}(\log p)^\nu\int_{2}^\infty \frac{1}{x^{\sigma}\log x} dx + \frac{(\nu-1)!}{ (\sigma-1)^\nu}\frac{p\sigma}{p^\sigma}\sum_{i=0}^{\nu-1}\frac{\nu}{\nu-i}\frac{\left((\sigma-1)\log p\right)^{i}}{i!},
\end{align*}
where we have used the identity \[\int_{1}^\infty \frac{(\log x)^a}{x^\sigma}dx= \int_{0}^\infty \frac{t^{a}}{e^{(\sigma-1)t}}dt = \frac{a!}{(\sigma-1)^{a+1}}.\]

We consider first the term \begin{align*}
\frac{p\sigma}{p^\sigma}(\log p)^\nu\int_{2}^\infty \frac{1}{x^{\sigma}\log x} dx &= \frac{p\sigma}{p^\sigma}(\log p)^\nu\int_{\log2}^\infty e^{-(\sigma-1)t} \frac{dt}{t}\\
& \leq\frac{p\sigma}{p^\sigma}(\log p)^\nu \left(\int_{(\sigma-1)\log2}^1 \frac{1}{t} dt+\int_{1}^\infty e^{-t} dt\right)\\
& \leq (\log p)^\nu\left(\log(\frac{1}{\sigma-1})-\log_2(2) + e^{-1}\right).
\end{align*} 
Because $p\sigma\leq p^\sigma$. We now seek the $\varepsilon_2$ such that \[(\log p)^\nu\left(\log(\frac{1}{\sigma-1})-\log_2(2) + e^{-1}\right) \leq \varepsilon_2 \frac{(\nu-1)!}{(\sigma-1)^\nu}.\] If we put $\varepsilon_2=\frac{\log_2(p)+\log(c)-\log_2(2) + e^{-1}}{c^\nu (\nu-1)!}$, the inequality holds for for $\sigma\rightarrow 1$ and for $\sigma=1+\frac{1}{c\log p}$. One may check that the derivative of the difference does not have a zero in the interval under consideration if $p>e^e$. Thus the difference is monotone, and the inequality holds throughout.

To deal with the rest of the terms efficiently, write $X=(\sigma-1)\log p \leq 1/c$. Then we have for any integer $B\geq 1$ 
\begin{align*}
\frac{p\sigma}{p^\sigma}\sum_{i=0}^{\nu-1}\frac{\nu}{\nu-i}\frac{X^{i}}{i!} & \leq 
\frac{p\sigma}{p^\sigma}\sum_{i=0}^{B-1}\frac{\nu}{\nu-B}\frac{X^{i}}{i!} + \frac{p\sigma}{p^\sigma}X^B\sum_{i=0}^{\nu-1}\frac{\nu}{B!}\frac{X^{i-B}}{(i-B)!} \\
& \leq 
\frac{p\sigma}{p^\sigma}\frac{\nu}{\nu-B}e^X + \frac{p\sigma}{p^\sigma}\frac{\nu}{c^\nu B!}e^X = \frac{\nu \sigma}{\nu-B} + \frac{\nu \sigma}{c^B B!}
\end{align*}
We now put $B=\lfloor \log \nu \rfloor$, and see that the to be bounded sum is bounded by $(1 + \varepsilon_3)\frac{(\nu-1)!}{(\sigma-1)^\nu}$, where $\varepsilon_3 = \frac{1}{c\log p} + \frac{\sigma \lfloor \log \nu \rfloor}{\nu - \lfloor \log \nu \rfloor} + \frac{\sigma\nu}{c^{\lfloor \log \nu \rfloor} \lfloor \log \nu \rfloor !}$

One may now bound the $\varepsilon_1 + \varepsilon_2 + \varepsilon_3$ by the coefficient of $\frac{(\nu-1)!}{(\sigma-1)^\nu}$ except the $1_\beta$ in the statement of the lemma.
We note that the sum over the primepowers congruent to $-1$ mod $p$ obeys the same bound with the same proof as above. One of the sums is strictly positive and the other is strictly negative, thus we have proven that \[|f^{\nu}(s)+\left(\log(\sigma-\beta)\right)^{(\nu)}|\leq (1+c_{p,\nu})\frac{(\nu-1)!}{(\sigma-1)^\nu},\] or since $\frac{(\nu-1)!}{(\sigma-\beta)^\nu}\leq \frac{(\nu-1)!}{(\sigma-1)^\nu}$, \[|f^{\nu}(s)|\leq (1+1_\beta+c_{p,\nu})\frac{(\nu-1)!}{(\sigma-1)^\nu}.\]
\end{proof}

On the other hand  we can prove the following bound on the derivatives of $f$ to the right of $s=1$, using the holomorphic property of $f$ on $B(1, \frac{1}{c\log p})$, when $c$ is big enough. We note that due to Kadiri (\cite{Kadiri}, Theorem 12.1) the value $c=6.4355$ is big enough.
\begin{lem}For $c>6.4355$, $\frac{p-1}{\log p}>c$, and $\sigma \in [1, 1+ \frac{2}{c\log p}]$, we have that \begin{eqnarray}\label{eq:absolute}|f^{(\nu)}(\sigma)|\leq 2c^\nu\nu!\; p\log^{\nu+1}p\end{eqnarray}
\end{lem} 
\begin{proof} 
Recall the lemma of Borel-Caratheodory (see \cite{Estermann}, p. 12) which states that if $g$ is holomorphic and $\Re(g(s))\leq M$ in $B(\sigma_0, R)$ and $g(\sigma_0)=0$, then \[|g^\nu(s)|\leq \frac{2M\nu!}{(R-r)^\nu},\quad \quad s\in B(\sigma_0, r).\]
We wish to apply this to $f(s)-f(\sigma_0)$. This function vanishes at $\sigma_0$, and is holomorphic as long as $R \leq \sigma_0-(1-\frac{1}{c\log p}).$ For the bound on the real part, consider
\[L(s,\chi)=\sum_{n=1}^\infty \frac{\chi(n)}{n^s} = s\int_1^\infty \frac{\sum_{n\leq x} \chi(n)}{x^{s+1}}dx.\] Since $|\sum_{n=1}^x \chi(n)|\leq \frac{p}{2}$, we have that $|L(s, \chi)|\leq |s|\int_1^\infty \frac{|\sum_{n\leq x} \chi(n)|}{x^{\sigma+1}}dx\leq \frac{|s|p}{2\sigma}$. This means that \[\Re(f(s))\leq \frac{p-1}{2}(\log p+\log(|s|/2\sigma)) - \log(|s-\beta|).\]
For $s$ on the border of the domain determined by $3/4<\Re(s)<2, |\Im(s)|\leq \frac{1}{4}$, $|s|/2\sigma \leq \sqrt{10}/6$ and say $|s-\beta|>1/8$, thus this bound is smaller than $\frac{p-1}{2}\log p$. Since $f(s)$ is harmonic with at most logarithmic singularities in which $\Re(f)\rightarrow -\infty$, the same bound holds also inside the domain.
In the region $\sigma>1$, consider the following estimation. \[|\Re(\log L(s, \chi))|=|\Re\Big(\sum_{q^m} \frac{\chi(q^m)}{mq^{ms}}\Big)|\leq \sum_{q^m} \frac{1}{mq^{ms}} = \log \zeta(\sigma) \leq \log(\frac{\sigma}{\sigma-1}),\] thus if $\sigma_0>p/(p-1)$, then $|\Re(f(\sigma_0))|\leq  \frac{p-1}{2}\log(p) + \log(p-1).$ In conclusion, as long as $\sigma_0>p/(p-1)$, \[\Re(f(\sigma)-f(\sigma_0))\leq p\log p.\]
One retrieves the statement of the theorem by putting $\sigma_0 = 1+\frac{1}{c\log p}, R = \frac{2}{c\log p}, r = \frac{1}{c\log p}$.
\end{proof}

\section{Worst case scenario}\label{WorstCase}

Among all functions that satisfy the bounds from the preceding section, what is the largest value $f(1)$ can attain? We define $\sigma_\nu$ to be the point where the bound \eqref{eq:diverging} and the absolute bound \eqref{eq:absolute} coincide. We note that \begin{align}\label{eq:sigmanu}\sigma_\nu-1=\frac{1}{c\log p}\sqrt[\nu]{\frac{1+1_\beta+c_{p,\nu}}{2\nu p\log p}} \geq \frac{1}{c\log p\sqrt[\nu]{2\nu p\log p}}.\end{align}
\newpage

\begin{thm} For all $p>500$, and $c>6.4355$,  \[|f(1)|\leq (1+1_\beta.2+e^{1/c})\log_2(p) + O(1),\] where the $O(1)$-term is bounded by $(3+e^{1/c})\log(c) + 0.791e^{1/c} + 10.720 + \frac{0.943}{c}$
\end{thm}
\begin{proof} 
We use the Taylor expansion of $f$ with error term in integral form,
\begin{align*}
f(1) &= f(\sigma_\nu) + (1-\sigma_\nu)f^{'}(\sigma_\nu)+\frac{(1-\sigma_\nu)}{2}^2f^{(2)}(\sigma_\nu) + ...\\ 
&\quad+ \int_{\sigma_\nu}^1 \frac{f^{(\nu)}(x)}{(\nu-1)!} (1-x)^{\nu-1}dx
\end{align*}
Now note that $|f^{(\nu)}(x)|$ is bounded above by the bound \eqref{eq:absolute} for all $x$ between $1$ and $\sigma_\nu$, which is equal to $|f^{(\nu)}(\sigma_{\nu})|$. Using \eqref{eq:divergingbasic}, \eqref{eq:diverging} and \eqref{eq:sigmanu}, we get 
\begin{align*} |f(1)| &\leq |f(\sigma_\nu)|+\sum_{i=1}^\nu \frac{(\sigma_\nu-1)^i}{i!}|f^{(i)}(\sigma_\nu)|.\\
&\leq (1+1_\beta)\log(\frac{1}{\sigma_\nu-1}) + 3/2 + \sum_{i=1}^\nu \frac{1+1_\beta + c_{p,i}}{i}\\
&\leq (1+1_\beta)\Big(\log_2(p) + \log(c)+\frac{\log(2\nu p\log p)}{\nu}\Big) + 3/2 + \sum_{i=1}^\nu \frac{1+1_\beta + c_{p,i}}{i}.
\end{align*}
Upon taking $\nu=\log p$, this first contribution is bounded by \[(1+1_\beta)\Big(\log_2(p) + \log(c)+1+\frac{\log(2(\log p)^2)}{\log p}\Big)+3/2.\]
In the rest of the terms, we find the first $\nu$ terms of some converging series; \[\sum_{i=1}^\nu \frac{1}{c^i i!}\leq e^{1/c} - 1, \quad \sum_{i=1}^\nu \frac{\lfloor \log \nu \rfloor}{\nu(\nu - \lfloor \log \nu \rfloor)}\leq 1.90, \quad \sum_{i=1}^\nu \frac{1}{c^{\lfloor \log \nu \rfloor} \lfloor \log \nu \rfloor !}\leq 1.13.\]
Using this and the well-known estimate $\sum_{i=1}^\nu \frac{1}{i} \leq \log(\nu) + 1$ we bound the last contribution as follows 
\begin{align*} \sum_{i=1}^\nu \frac{1+1_\beta + c_{p,i}}{i} &\leq (1+1_\beta+\frac{1}{c\log p})(\log(\nu)+1) + (1+\frac{1}{c\log p})3.03\\
& + \Big(\frac{\log(2)}{2\log p} + \log_2(p) + \log(c) - \log_2(2) + e^{-1}\Big)(e^{1/c}-1).
\end{align*}
Gathering everything and filling in $p=500$ for the terms converging to zero, we recover the statement of the theorem.
\end{proof}

To finish the proof of Theorem \ref{Theorem1}, one now needs to plug the above estimate of $f(1)$ into the logarithm of the formula \eqref{eq:classnumberformula}, and check that the choice of $c=\log_2(p)\frac{6.4355}{\log_2(500)}$ is permitted. 
\begin{rem} It is quite counterintuitive that a bigger value of $c$ gives a better estimate in Theorem 3.1 while a smaller value of $c$ means a bigger zero-free region, thus means a stronger input. In truth there is a tradeoff between having $\sigma_\nu$ \itshape big \upshape to control the main term coming from Lemma 2.2 and at the same time \itshape not too big \upshape to bound the term coming from $\varepsilon_1$ in the proof of Lemma 2.2. This $\varepsilon_1$ cannot be efficiently bounded by a lack of good bounds on the number of primes of the form $ap+1$, where $a$ is a small integer.
\end{rem}

\begin{rem}From \eqref{eq:classnumberformula} it is now clear that the general behaviour of $h_p^-$ is dominated by $G(p)$ and that the $L$-values can perturb this term only slightly. It is somewhat common(see \cite{Louboutin}) to state upper bounds for $h_p^-$ in terms of $G(p)$, where $4\pi^2$ is replaced by a smaller constant.
\end{rem}
\begin{cor} We have that $h_p^-\leq 2p\left(\frac{p}{39}\right)^\frac{p-1}{4}$, for $p>9649$.
\end{cor}
\begin{proof}This follows from plugging in $c=\frac{6.4355\log_2(p)}{\log_2(500)}=3.523\log_2(p)$ in Theorem 3.1 and checking that \[|f(1)|\leq e^\frac{p-1}{4}\log(\frac{4\pi^2}{39}),\]
whenever $p>9649$.
\end{proof}
\nocite{Estermann}
\bibliographystyle{plain}
\bibliography{biblio}

\end{document}